\documentclass[12pt]{amsart}
	
	\usepackage[a4paper,margin=3cm]{geometry}
	
	\usepackage{amsmath,amssymb,amscd,mathtools}
	\usepackage[mathscr]{eucal}
	
	\usepackage{enumitem}
	\usepackage{xspace}
	\usepackage{ifthen}
	
	\usepackage[all]{xy}
	\xyoption{poly}
	
	\usepackage{xcolor}
	
	\usepackage[
	hyperindex,
	bookmarksnumbered,
	plainpages,
	backref
	]{hyperref}
	
	\newtheorem{theorem}{Theorem}
	\newtheorem{lemma}{Lemma}

	\theoremstyle{definition}
	
	\newtheorem{example}{Example}

	\makeindex
	
		\title{When Is a Bogolyubov Automorphism Inner?}

\begin{document}

	\author{NIKITA ARSKYI} 
		\address{NIKITA ARSKYI: Faculty of Mechanics and Mathematics, \\ Taras Shevchenko  National University of Kyiv, \\  60	Volodymyrska Street, Kyiv 01033, Ukraine}
		\email{nikitaarskyi@knu.ua}
		
	\author{OKSANA BEZUSHCHAK}
		\address{OKSANA BEZUSHCHAK: Faculty of Mechanics and Mathematics, \\ Taras Shevchenko  National University of Kyiv, \\  60	Volodymyrska Street, Kyiv 01033, Ukraine}
		\email{bezushchak@knu.ua}

		\begin{abstract}	Let $V$ be an infinite-dimensional vector space over a field of characteristic not equal to $2$. Given a nondegenerate quadratic form $f$  on $V$, we consider the Clifford algebra $\mathrm{Cl}(V,f)$. Any orthogonal linear transformation of $V$  extends to a Bogolyubov automorphism of $\mathrm{Cl}(V,f)$. We obtain necessary and sufficient conditions for a Bogolyubov automorphism to be inner.	\end{abstract}
		
			\subjclass[2020]{15A66 (Primary) 16W20 (Secondary)}
		
			\keywords{Bogolyubov automorphism, Clifford algebra,    locally matrix algebra}

		\maketitle
		
		\section{Introduction}

Let $\mathbb{F}$ be a field of characteristic not equal to $2$, and let $V$ be a vector space over $\mathbb{F}$.
A mapping $f \colon V \to \mathbb{F}$ is called a \textbf{quadratic form} if the following conditions are satisfied:
\begin{enumerate}
	\item[(1)]	$f(\lambda v) = \lambda^2 f(v), \qquad \lambda \in \mathbb{F},\ v \in V;$
	\item[(2)]	$(v \mid w) = f(v+w) - f(v) - f(w)$	is a bilinear form  on $V$.
\end{enumerate}

A quadratic form is said to be  \textbf{nondegenerate} if the the associated bilinear form $(v \mid w)$ is nondegenerate, that is,  $(v \mid V) = (0) $  implies $ v = 0.$

The  \textbf{Clifford algebra} $\mathrm{Cl}(V,f)$ is the associative algebra generated
by the vector space $V$ and the identity element $1$ with defining
relations $v^2 = f(v)\cdot 1 $ for all $ v \in V.$ If the quadratic form $f$ is fixed, we simply write $\mathrm{Cl}(V)$.

If $\{v_i\}_{i \in I}$ is a basis of $V$ and the index set $I$ is totally
ordered, then the set consisting of the identity $1$ and all ordered products
\[
v_{i_1} v_{i_2} \cdots v_{i_k},
\qquad
i_1 < i_2 < \cdots < i_k,
\]
forms a basis of the Clifford algebra $\mathrm{Cl}(V)$.

The Clifford algebra $\mathrm{Cl}(V)$ is graded by the cyclic group of order $2$: 
\[
\mathrm{Cl}(V) = \mathrm{Cl}(V)_{\overline{0}} + \mathrm{Cl}(V)_{\overline{1}},
\]
where
\[
\mathrm{Cl}(V)_{\overline{0}}
=
\mathbb{F}\cdot 1 
+
\sum_{n=1}^{\infty}
\underbrace{V \cdots V}_{2n},
\qquad
\mathrm{Cl}(V)_{\overline{1}}
=
\sum_{n=1}^{\infty}
\underbrace{V \cdots V}_{2n-1}.
\]

Let $\varphi \colon V \to V$ be an orthogonal linear transformation, that is, $f(\varphi(v)) = f(v) $ for all $ v \in V.$ It is well-known that $\varphi$ admits a unique extension to an automorphism
$[\varphi]$ of the Clifford algebra $\mathrm{Cl}(V)$.
Such automorphisms are called  \textbf{Bogolyubov automorphisms}.

There is extensive literature on automorphisms of Clifford algebras over Hilbert and pre-Hilbert spaces, as well as on $\mathbb{C}^*$-completions of these
algebras (see, for example,~\cite{1,7,9}).
In particular, H.~Araki~\cite{1} described Bogolyubov automorphisms of Clifford
$\mathbb{C}^*$-algebras that are inner automorphisms (see also~\cite{9}).

Let $\varphi \colon V \to V$ be a linear transformation.
For an eigenvalue $\alpha$ of $\varphi$, we denote by
\[
V(\alpha) = \ker(\varphi - \alpha\cdot \mathrm{Id})
\]
the corresponding eigenspace, where $\mathrm{Id}$ denotes the identity
operator on $V$.

A linear transformation $\varphi$ is called \textbf{finitary} if the subspace
$V(1)$ has finite codimension in~$V$.

Let $W \subset V$ be a $\varphi$-invariant subspace of finite codimension.
Then $\varphi$ gives rise to the linear transformation $\varphi | _{V/W}$ on the finite-dimensional quotient space $V/W$.

In this paper, we describe inner Bogolyubov automorphisms of the Clifford
algebra $\mathrm{Cl}(V)$ in a purely algebraic setting.

\begin{theorem}\label{Th_1} Let $\mathbb{F}$ be an algebraically closed field of characteristic  $ \neq 2$, and let $V$ be an infinite-dimensional vector space over $\mathbb{F}$ with a nondegenerate quadratic form. Let $\varphi$ be an orthogonal linear transformation. The Bogolyubov automorphism $[\varphi]$  of the Clifford
	algebra $\mathrm{Cl}(V)$  is inner  if and only if 
\begin{enumerate}
	\item[$(1)$] $\varphi$ is finitary and either $\varphi=\mathrm{Id}$ or	$\det\!\left( \varphi | _{V/V(1)}  \right) = 1,$

	or 
	
	\item[$(2)$] $-\varphi$ is finitary, $ \varphi \not=-  \mathrm{Id},$ $\dim_{\mathbb{F}}(V/V(-1))=k \ge 1, $ and
	\[
	(-1)^k
	\cdot
	\det\!\left( \varphi \big|_{V/V(-1)} \right) = -1 .
	\]
\end{enumerate}
\end{theorem}

\begin{example}
Let $V$ be countable-dimensional and let $v_1, v_2, \dots$ be an orthonormal basis of $V$.
The linear transformation $\varphi$, defined by
\[
\varphi(v_1) = v_1, \quad
\varphi(v_i) = -v_i \quad  \text{for } i \ge 2,
\]
is orthogonal. We have $\dim_{\mathbb{F}}(V/V(-1))= 1 $ and $	\det\!\left( \varphi \big|_{V/V(-1)} \right) = 1 .$ For an arbitrary element $a \in \mathrm{Cl}(V)$, we have
\[
[\varphi](a) = v_1^{-1} a v_1 .
\] \end{example}

In~\cite{6}, we describe Bogolyubov derivations of $\mathrm{Cl}(V)$ that are inner.

The Clifford algebra $\mathrm{Cl}(V)$ is a unital locally matrix algebra
corresponding to the Steinitz number $2^{\infty}$ (see \cite{4,3,2}). Automorphisms and derivations of arbitrary  countable-dimensional
unital locally matrix algebras were described in~\cite{5}.

\section{Proof of the Lemmas} 
We call a finite-dimensional subspace $U \subset V$
\textbf{nondegenerate} if the restriction of the quadratic form to $U$
is nondegenerate.

\begin{lemma}\label{L_1}
\begin{enumerate}
	\item[$(1)$]
	Suppose that $V(1)$ has finite codimension in $V$.
	Then there exists an even-dimensional, nondegenerate,
	$\varphi$-invariant subspace $W \subset V$ such that
	\[
	W^{\perp}
	=
	\{ v \in V \mid (v \mid W) = (0) \}
	 \subseteq V(1).
	\]
	\item[$(2)$]
	Suppose that $V(-1)$ has finite codimension in $V$.
	Then there exists an even-dimensional, nondegenerate,
	$\varphi$-invariant subspace $W \subset V$ such that
	\[
	W^{\perp} =	\{ v \in V \mid (v \mid W) = (0) \}
	\subseteq V(-1).
	\]
\end{enumerate} \end{lemma}

\begin{proof} We prove only part~(1). Part~(2) is proved similarly. The subspace $(\mathrm{Id}-\varphi)(V)$ is finite-dimensional.
Let $\chi(t)$ be the characteristic polynomial of the restriction of
$\varphi$ to the subspace $(\mathrm{Id}-\varphi)(V)$.
Then $\varphi$ is a root of the polynomial $(1-t)\chi(t)$.

There exists a finite-dimensional $\varphi$-invariant subspace
$U \subseteq V$ such that $V = U + V  (1).$
For arbitrary elements $v_1, v_2 \in V$, we have
\[
(\varphi(v_1) \mid v_2) = (v_1 \mid \varphi^{-1}(v_2)).
\]  
Hence,
\[
\bigl((\varphi - \mathrm{Id})(v_1) \mid v_2\bigr)
=
\bigl(v_1 \mid (\varphi^{-1} - \mathrm{Id})(v_2)\bigr)
=
\bigl(v_1 \mid \varphi^{-1}(\mathrm{Id}-\varphi)(v_2)\bigr).
\]
This implies that $\bigl( (\varphi - \mathrm{Id})(V)  \mid     V(1)\bigr)=(0).$

Let $v \in U^{\perp}$ and suppose that $(\varphi - \mathrm{Id})(v) \neq 0$.
Then $(\varphi - \mathrm{Id})(v)$ is orthogonal both to $U$ and to $V(1)$.
This contradicts the nondegeneracy of the quadratic form on $V$.
Therefore, we have proved that $U^{\perp} \subseteq V(1).$

Suppose that the subspace $U$ is nondegenerate.
If $\dim_{\mathbb{F}} U$ is even, then we are done.
Assume that $\dim_{\mathbb{F}} U$ is odd.
Since the restriction of the quadratic form to $U^{\perp} $ is nondegenerate, there exists an element $v \in U^{\perp} $ such that $(v \mid v) \neq 0$. The subspace $U + \mathbb{F}v $ satisfies all the required assumptions.

Now suppose that the subspace $U$ is not nondegenerate.
Let $0 \neq u \in U \cap U^{\perp} \subseteq U \cap V(1)$.
Since the quadratic form on $V$ is nondegenerate, there
exists an element $v \in V(1)$ such that $(u \mid v) \neq 0.$ 
The subspace $\mathbb{F} u + \mathbb{F} v$ is nondegenerate, and we obtain the orthogonal decomposition $V = (\mathbb{F} u + \mathbb{F} v) \oplus (\mathbb{F} u + \mathbb{F} v)^{\perp}.$

Consider the projection $\pi \colon V \to V' = (\mathbb{F} u + \mathbb{F} v)^{\perp}.$ We have $V' = \pi(U) + V'(1),$ where the subspace $\pi(U)$ is $\varphi$-invariant and $\dim_{\mathbb{F}} \pi(U) < \dim_{\mathbb{F}} U.$ Using induction on $\dim_{\mathbb{F}} U$, we find an even-dimensional
$\varphi$-invariant nondegenerate subspace $W' \subset V'$ such that $(W')^{\perp} \cap V' \subseteq V'(1).$ It remains to choose $W = \mathbb{F} u + \mathbb{F} v + W'.$

This completes the proof of the lemma. \end{proof}

For a subset $S$ of an algebra $A$, let
$$Z_A(S)=\{a\in A\ |\ ax=xa \text{ for each } x\in S\}$$
be the centralizer of the set $S$ in $A$.

\begin{lemma}\label{L_2}
Let $U$ be a nondegenerate even-dimensional subspace of $V$. Let $v_1,\ldots,v_n$ be an orthonormal basis of $U$. Then
$$Z_{\mathrm{Cl}(V)}(\mathrm{Cl}(U))=\mathrm{Cl}(U^{\perp})_{\overline{0}}+v_1\cdots v_n\mathrm{Cl}(U^{\perp})_{\overline{1}}.$$
\end{lemma}
\begin{proof}
Let $\{u_i\}_{i\in I}$ be an ordered basis of $U^{\perp}$. Consider an element
$$x=v_{i_1}\cdots v_{i_k}u_{j_1}\cdots u_{j_l},$$
where $1\leq i_1<\cdots<i_k\leq n$, $j_1<\ldots<j_l$. For each $i\notin\{i_1,\ldots,i_k\}$,
$$v_ix=(-1)^{k+l}xv_i.$$
For each $t\in\{1,\ldots,k\}$,
$$v_{i_t}x=(-1)^{k+l-1}xv_{i_t}.$$
Hence, the element $x$ will commute with each of the elements $v_1,\ldots,v_n$ if and only if either $k=0$ and $l$ is even, or $k=n$ and $l$ is odd. Otherwise, $x$ will anticommute with $v_i$ for some $i$. This completes the proof of the lemma.
\end{proof}

\begin{lemma}\label{L_3}	 	Let $\varphi \colon V \to V$ be an orthogonal linear transformation, and let $[\varphi]$ be the corresponding Bogolyubov automorphism.	Let $v_1, v_2 \in V$ be linearly independent.	If	$[\varphi](v_1 v_2) = v_1 v_2,$	 	then the subspace $\mathbb{F} v_1 + \mathbb{F} v_2$ is $\varphi$-invariant.
	 \end{lemma}
\begin{proof} 
Suppose that $\varphi(v_1)\notin\mathbb{F} v_1 + \mathbb{F} v_2$. Extend the system $v_1,v_2$ to a basis $v_1,v_2,v_3,\ldots$ of $V$. Let
$$\varphi(v_1)=\sum_{i=1}^{n}{\alpha_iv_i} \quad \text{and} \quad \varphi(v_2)=\sum_{i=1}^{n}{\beta_iv_i}.$$
Without loss of generality, let $\alpha_3\neq0$. We have that
\begin{gather*}
v_1v_2=[\varphi](v_1v_2)=(\alpha_1\beta_2-\alpha_2\beta_1)v_1v_2+\sum_{i\leq j}{\alpha_j\beta_if(v_i,v_j)1}+\\
+\sum_{\substack{i<j\\(i,j)\neq(1,2)}}{(\alpha_i\beta_j-\alpha_j\beta_i)v_iv_j}.
\end{gather*}

Since the system of vectors $\{1;v_iv_j,i<j\}$ is linearly independent, it follows that $\alpha_1\beta_2-\alpha_2\beta_1=1$, $\alpha_1\beta_3=\alpha_3\beta_1$ and $\alpha_2\beta_3=\alpha_3\beta_2$. However, the last two equalities contradict the first one. Hence, $\varphi(v_1)\in\mathbb{F}v_1+\mathbb{F}v_2$. Analogously, $\varphi(v_2)\in\mathbb{F}v_1+\mathbb{F}v_2$, which completes the proof of the lemma.
\end{proof}

\section{Proof of the Theorem} 

\begin{proof}[Proof of Theorem $\ref{Th_1}$] \textit{Necessity.} Let $[\varphi]$ be a conjugation by an invertible element $x$. There exists a nondegenerate even-dimensional subspace $V'\subseteq V$ such that $x\in\mathrm{Cl}(V')$. 

Choose an arbitrary vector $0\neq v\in V'^{\perp}$. Choose arbitrary elements $v_1,v_2 \in V'^\perp$ such that the system $v,v_1,v_2$ is linearly independent. The elements $v v_1$ and $v v_2$ lie in $Z_{\mathrm{Cl}(V)}(\mathrm{Cl}(V'))$.
	Hence,
	\[
	[\varphi](v v_1) = v v_1,
	\qquad
	[\varphi](v v_2) = v v_2.
	\]
	By Lemma~\ref{L_3}, the subspaces $\mathbb{F}  v + \mathbb{F}  v_1$ and $\mathbb{F}  v + \mathbb{F}  v_2$ are
	$\varphi$-invariant. Hence,
	\[
	\varphi(v) \in (\mathbb{F}  v + \mathbb{F}  v_1) \cap (\mathbb{F}  v + \mathbb{F}  v_2) = \mathbb{F}  v.
	\]
	We have shown that any element $v \in V'^\perp$ is an eigenvector of $\varphi$, with
	\[
	\varphi(v) = v \quad \text{or} \quad \varphi(v) = -v.
	\]
	
	Hence, $V'^\perp \subseteq  V(1)$ or $V'^\perp \subseteq  V(-1)$. Consider the first case. The second case is considered similarly.

    Let $v_1,\ldots, v_n$ be an orthonormal basis of $V'$. It follows from Lemma \ref{L_2} that for each $u\in V'^{\perp}$,
    $$[\varphi](v_1\cdots v_nu)=v_1\cdots v_nu.$$
    We have that
    $$[\varphi](v_1\cdots v_nu)=\det\varphi|_{V/V(1)}v_1\cdots v_nu.$$

    Thus, $\det\varphi|_{V/V(1)}=1$.

    \textit{Sufficiency.} Suppose condition $(1)$ of Theorem \ref{Th_1} is satisfied. Case $(2)$ is considered similarly. By Lemma \ref{L_1}, there exists a nondegenerate, even-dimensional, $\varphi$-invariant subspace $V'$ of $V$ such that $V'^{\perp}\subseteq V(1)$. We have that $\mathrm{Cl}(V)=\mathrm{Cl}(V')\otimes Z_{\mathrm{Cl}(V)}(\mathrm{Cl}(V'))$; see \cite{8,10}.

    Let $v_1,\ldots,v_n$ be an orthonormal basis of $V'$. For each linearly independent $u_1,u_2\in V'^{\perp}$,
    $$[\varphi](v_1\cdots v_nu_1)=\det\varphi|_{V/V(1)}v_1\cdots v_nu_1=v_1\cdots v_nu_1\quad\text{and}\quad[\varphi](u_1u_2)=u_1u_2.$$
    It follows from Lemma \ref{L_2} that $[\varphi]$ acts identically on $Z_{\mathrm{Cl}(V)}(\mathrm{Cl}(V'))$. Thus $[\varphi]$ must be a conjugation with an element from $\mathrm{Cl}(V')$.
    
    This completes the proof of Theorem~\ref{Th_1}. \end{proof}

 	\section*{Acknowledgments}

 \medskip
 
The authors are deeply grateful to Efim Zelmanov for valuable discussions.

\end{document}